\documentclass{amsart}

\usepackage{graphicx}%
\usepackage{multirow}%
\usepackage{amsmath,amssymb,amsfonts}%
\usepackage{amsthm}%
\usepackage{mathrsfs}%
\usepackage[title]{appendix}%
\usepackage{xcolor}%
\usepackage{textcomp}%
\usepackage{manyfoot}%
\usepackage{booktabs}%
\usepackage{algorithm}%
\usepackage{algorithmicx}%
\usepackage{algpseudocode}%
\usepackage{listings}%
\usepackage{longtable}
\usepackage{nicematrix,tikz}
\usepackage{subcaption}
\usepackage{natbib}
\setcitestyle{authoryear,open={(},close={)}} %Citation-related commands

\newtheorem{theorem}{Theorem}[section]
\newtheorem{proposition}[theorem]{Proposition}
\newtheorem{lemma}[theorem]{Lemma}

\theoremstyle{definition}

\theoremstyle{remark}

\numberwithin{equation}{section}

\begin{document}

%\title[On the largest and smallest roots]{On the distribution of the largest and smallest roots of the ratio}

\title{On incomplete Gamma and Beta integrals}

%    author one information
\author[Haoming Wang]{Haoming Wang}
\address{School of Mathematics, Sun Yat-sen University, Guangzhou {\rm510275}, China}
\curraddr{Center for Combinatorics, Nankai University, Tientsin {\rm500071}, China}
\email{wanghm37@nankai.edu.cn}
\thanks{}

%    author two information
%\author[Jiang Xiaotian]{JIang Xiaotian}
%\address{}
%\curraddr{}
%\email{jian0851@umn.edu}
%\thanks{}

%    \subjclass is required.
\date{}

\dedicatory{}

%    Abstract is required.
\begin{abstract} This paper discusses the incomplete Gamma and Beta integrals involving the generalised hypergeometric function. The distribution of the largest and the smallest roots of a ratio arising in comparing the mean differences among groups is obtained as an application.
\end{abstract}

\maketitle

\section{Introduction}
The incomplete Gamma integral is closely related to the development of statistical distributions, such as \cite{Student1908}'s distribution of standard deviations, in samples from a normal population,
\begin{equation}
    p(a) = \frac{1}{\varGamma(\frac{n-1}{2})} A^{\frac{n-1}{2}} e^{-Aa} a^{\frac{n-3}{2}},
\end{equation}
where $n$ is the size of the sample, and
\[A= \frac{n}{2\sigma^2}, \quad a = s^2,\]
with $\sigma$ the standard deviation of the sampled population, and $s$ that estimated from the sample. Thus, if $x_1, x_2,\dots x_n$ are the sample values, 
\[\begin{aligned}
    n s^2 = \sum_{i=1}^{n} (x_i - \bar{x})^2, \quad n \bar{x} = \sum_{i=1}^{n} x_i.
\end{aligned}\]

When bivariate populations were considered, other problems arose, such as the distribution of the correlation coefficient and of the regression coefficient in samples. These problems, taken by themselves, were found to be difficult when \cite{Fisher1915} gave a formula for the simultaneous distribution of the three quadratic statistical derivatives, namely the two sample variances (squared standard deviation) and the sample covariances (product moment coefficient)
\begin{equation}
p(a,b,h)=\frac{1}{\pi^{\frac{1}{2}} \varGamma(\frac{n-1}{2}) \varGamma(\frac{n-2}{2})}\left|\begin{array}{ll}A & H \\ H & B\end{array}\right|^{\frac{n-1}{2}}e^{-A a-B b-2 H h} \cdot\left|\begin{array}{ll}a & h \\ h & b\end{array}\right|^{\frac{n-1}{2}},
\label{eq: product moment distribution}
\end{equation}
where
\begin{equation*}
    \begin{aligned}
    & A = \frac{n}{2\sigma_{1}^{2}(1-\rho^2)}, \,\quad B = \frac{n}{2\sigma_{2}^{2}(1-\rho^2)}, \,\quad H = \frac{n\rho}{2\sigma_{1}\sigma_{2}(1-\rho^2)},\\
& \, a = s_{1}^{2}, \qquad\qquad \qquad  b = s_{2}^{2}, \qquad \qquad \qquad  h = rs_{1}s_{2}.
\end{aligned}
\end{equation*}
where $\sigma_1, \sigma_2$ are the standard deviations of the sampled population and $\rho$ the correlation
between two population variables. 

Based on the works of \cite{Fisher1915}, \cite{Romanovsky1925}, and \cite{KPearson1925}, \cite{wishart1928generalised} further obtained the joint distribution of $p$ sample variances and $\frac{1}{2}p(p-1)$ sample covariances, leading to the product moment distribution in samples from a multivariate normal population, expressed in a neat matrix form
\begin{equation}
    p(S) = \frac{1}{2^{\frac{np}{2}}\varGamma_{p}(\frac{n}{2})|\varPsi|^{\frac{n}{2}}}\exp \left[-\frac{1}{2}\operatorname{tr}  (\varPsi^{-1}S)\right]|S|^{\frac{n-p-1}{2}}, \label{eq: central Wishart}
\end{equation}
where the probability density function tasks positive values when the $p\times p$ real symmetric matrix $S = (s_{ij})$ is positive definite and elsewhere zero, $\varPsi$ is the covariance of the sampled population, and
\[\varGamma_{p}(a) = \pi^{\frac{p(p-1)}{4}}\prod_{i=1}^{p} \varGamma \left(a - \frac{i-1}{2}\right),\]
namely the multivariate Gamma function, with the univariate Gamma function $\varGamma(\cdot)$ given by the integral
\[\varGamma(a) = \int_{0}^{\infty}e^{-z}z^{a-1}dz, \quad \Re(a) > 0. \]
%The multivariate Beta function, denoted by ${\rm B}_{p}(a,b)$, is defined to be
%\begin{equation}   {\rm B}_{p}(a,b) =\int_{0< X < I} |X|^{a - \frac{p+1}{2}} |I-X|^{b - \frac{p+1}{2}} dX,\end{equation}
%where $\Re (a), \Re (b) > \frac{1}{2}(p-1)$, and the integral is taken over all $n\times n$ real symmetric matrices $X$ such that both $X$ and $I - X$ are positive definite. The multivariate Beta function is related to the multivariate Gamma function by the formula 
%\begin{equation}    {\rm B}_{p}(a,b) = \frac{\Gamma_{p}(a)\Gamma_{p}(b)}{\Gamma_{p}(a + b)}.\end{equation}

When doing inference from such an ensemble of matrices, usually we need not the density function but the distribution of latent roots. Thus, \cite{constantine1963some, constantine1966hotelling}, inspired from the works of \cite{hotelling1931student}, \cite{Wilks1934}, \cite{hsu1939determinantal}, and \cite{james1961zonal, james1964distribution}, gave the distribution of the largest latent root of $S$ distributed according to \eqref{eq: central Wishart}, from the incomplete Gamma integral formula
\begin{equation}
    {P} (S < R) = \frac{\varGamma_{p}\left(\frac{n + p+1}{2}\right) }{\varGamma_{p}\left(\frac{n}{2}\right)\varGamma_{p}\left(\frac{p+1}{2}\right)}\left|({1}/{2})\varPsi^{-1}  R\right|^{\frac{n}{2}} {}_1F_{1} \left(\frac{n}{2}; \frac{n +p+1}{2}; -  \frac{1}{2}\varPsi^{-1}  R\right)
    \label{eq: incomplete gamma 63}
\end{equation}
where the probability \eqref{eq: incomplete gamma 63} is calculated from \eqref{eq: central Wishart} by integrating over where $R-S$ is positive definite. His result, in a word, expresses the incomplete Gamma integral as the confluent hypergeometric function $_{1}F_{1}$. 

Now we are confronted with the non-central problem that might also be of the same importance when the hypergeometric function is itself the integrand. As a natural extension to \eqref{eq: incomplete gamma 63}, \cite{Davis1979InvariantPW, Davis1981OnTC} gave the incomplete Gamma integral involving the generalised hypergeometric function $_{p}F_{q}$
\begin{equation}
    \begin{aligned}
        \int_{0}^{ R} \exp\left[-\frac{1}{2}\operatorname{tr}(\varPsi^{-1} S)\right]|S|^{\frac{n-p-1}{2}} {}_{p}F_{q}(\left(\begin{matrix}
            a_{1},  \dots,a_{p}\\
            b_{1},  \dots,b_{q}
        \end{matrix}; \varPhi^{-1} S\right) (d  S) \\
                = \frac{\varGamma_{p}\left(\frac{n+p+1}{2}\right) }{\varGamma_{p}\left(\frac{n}{2}\right)\varGamma_{p}\left(\frac{p+1}{2}\right)}| R|^{\frac{n}{2}} {}_{p+1}F_{q+1} \left(\begin{matrix}
              a, a_{1} ,\dots,a_{p}\\
            a + \frac{p+1}{2}, b_{1}, \dots,b_{q}
        \end{matrix};- \frac{1}{2}\varPsi^{-1} R, \varPhi^{-1} R\right)
    \end{aligned}
\end{equation}
This formula was then developed by \cite{Chikuse1981}, and simplified by the result of \cite{GUPTA2004671} to derive the exact distribution in the analysis of variance and its associated Behrens-Fisher problem.  Despite significant progress, such as \cite{Diaconis1994}, \cite{mssriv2003},  \cite{johnstone2008multivariate}, \cite{ramirez2011}, \cite{Dumitriu2012jacobi}, \cite{dubbs2014beta}, \cite{ChianiZanella2020}, \cite{mathai2022singular}, and \cite{ForresterKumar2023}, an explicit formula for the non-central distribution of the largest root is not well-known until present.

In this article, we are going to study the joint distribution of the $\frac{1}{2}p(p+1)$ quadratic forms and the joint distribution of the latent roots of $S = (s_{ij})$,
\begin{equation}
    \begin{aligned}
        {y}_1'{A}_{11}{y}_1(=s_{11}), {y}_1'{A}_{12}{y}_2(=s_{12}), \dots, {y}_1'{A}_{1p}{y}_p(=s_{1p}),\\
        {y}_2'{A}_{22}{y}_2(=s_{22}), \dots, {y}_2'{A}_{2p}{y}_p(=s_{2p}), \\
        \vdots\qquad \\
        {y}_p'{A}_{pp}{y}_p(=s_{pp}),
    \end{aligned}
    \label{eq: illustration of triangular ararry heter}
\end{equation}
where $ A_{ij}^{\prime} =  A_{ji}$ may indicate the inverted covariance matrix between the $i$th and the $j$th variables, and $Y=ZB$ with $B = (b_{1},b_{2},\dots,b_{p})$ being an orthogonal matrix and $\operatorname{vec}(Z')$ being an $n\times p$ normal vector. The joint distribution of the $\frac{1}{2}p(p+1)$ quadratic forms when $n > p - 1$, is found to be 
        \begin{eqnarray}
         \frac{| \varTheta |^{\frac{1}{2}}}{2^{\frac{np}{2}}\varGamma_{p}(\frac{n}{2})} \exp \left[-\frac{1}{2}  \operatorname{tr}(U T) \right] | T|^{\frac{n-p-1}{2}}    %      \text{ when } {M}=0; \label{eq: product moment distribution central}
                %\times \operatorname{etr} \left(-\frac{1}{2}  \varOmega\right){}_{0}F_{1}\left(\frac{n}{2};\frac{1}{4} \varDelta  T\right) \text{ when } {M}\neq0;\label{eq: product moment distribution non-central}
        \end{eqnarray}
        where $\varTheta = \sum_{i,j=1}^{p} A_{ij} \otimes (b_{i}b_{j}^{\prime})$, $\otimes $ being the Kronecher product, $U = (u_{ij})$ with $u_{ij} = \operatorname{tr} (A_{ij})$, and $T=(t_{ij})$ with $t_{ij} = b_{i}^{\prime}Sb_{j}$. It can be shown that %if $\operatorname{vec}(Z')$ is a multivariate normal distribution, then 
        $A_{jj^{\prime}}(k,k)  = 0 $ if $j \neq j^{\prime}$ if and only if $Z$ is partially diagonalisable, say there exist orthonormal column vectors $\{ b_{j}\}$ of length $p$ such that $Z$ can be developed into the sum
    	\[ Z = \sum_{j=1}^{p}  Z_{j}  {b_{j}}^{\prime},\eqno(T_{1})\]
    	where $ Z_{j} =  Zb_{j}$ and ${E} Z_{j}(i) { Z_{j^{\prime}}^{\prime}(i)} = c_{j}(i)\delta_{jj^{\prime}}$. Rather than this, there are three additional stochastic representations $T_{1\frac{1}{2}}$, $T_2$, $T_3$, often useful in establishing the non-central distribution of the largest root and related theorems. These results include
        \begin{enumerate}
            \item[(A)] The central and non-central distribution of quadratic forms,
            \item[(B)] The latent roots with unequal covariances, and
            \item[(C)] The non-central latent roots with unknown covariance,
        \end{enumerate}
        which extends classical results in multivariate analysis of variance based on incomplete Gamma and Beta integrals involving hypergeometric functions.
        %Based on these results, the Tracy-Widom distribution seems achievable. 
        
        %\cite{johansson1997random}, \cite{baik1999distribution}, \cite{johnstone2001distribution}

\section{Incomplete Gamma and Beta Integrals Involving Hypergeometric Functions}
Let $\operatorname{etr} = \exp \operatorname{tr}$. The multivariate Gamma function, denoted by $\varGamma_{p}(a)$, is defined to be
\begin{equation}
    \varGamma_{n}(a) = \int_{A>0} \operatorname{etr} (-A) |A|^{a-\frac{n+1}{2}} (dA),
\end{equation}
where $\Re (a) > \frac{1}{2}(n-1)$ and $A>0$ means the integral is taken over the space of real symmetric positive definite $n \times n$ matrices. 

The multivariate Beta function, denoted by $B_{n}(a,b)$, is defined to be
\begin{equation}
    B_{n}(a,b) =\int_{0< X < I} |X|^{a - \frac{n+1}{2}} |I-X|^{b - \frac{n+1}{2}} (dX),
\end{equation}
where $\Re (a), \Re (b) > \frac{1}{2}(n-1)$, and the integral is taken over all $n\times n$ real symmetric matrices $X$ such that both $X$ and $I - X$ are positive definite. The multivariate Beta function is related to the multivariate Gamma function by the formula 
\begin{equation}
    B_{n}(a,b) = \frac{\varGamma_{n}(a)\varGamma_{n}(b)}{\varGamma_{n}(a + b)}.
\end{equation}

The hypergeometric function of a matrix argument is defined by a recursive relation due to \cite{herz1955}
	\begin{equation}
	    \begin{aligned}\int_{ X >  0} \operatorname{etr}(- X Z) | X|^{a - \frac{n+1}{2}} &{}_{p}F_{q} \left(\begin{matrix}    a_{1}, \dots, a_{p}\\ b_{1}, \dots, b_{q}    	\end{matrix};  X\right) (d X)\\		= \varGamma_{n}(a) | Z|^{-a} &{}_{p+1}F_{q} \left(\begin{matrix}    a_{1}, \dots, a_{p},a\\ b_{1}, \dots, b_{q}\quad    	\end{matrix};  Z^{-1}\right) 	\end{aligned}\label{eq: gamma integral}\\
	\end{equation}
	\begin{equation}
	    \begin{aligned}
    	%\int_{ 0<  X <  I} | X|^{a - \frac{n+1}{2}} | I- X|^{b - \frac{n+1}{2}} {}_{p}F_{q} \left(\begin{matrix}    a_{1}, \dots, a_{p}\\ b_{1}, \dots, b_{q}    	\end{matrix};  T X\right) d X\\		= \frac{1}{B_{n} (b-a,a)}{}_{p+1}F_{q+1} \left(\begin{matrix}    	    a, a_{1}, \dots, a_{p}\\ b, b_{1}, \dots, b_{q}    	\end{matrix};  T\right),\\
        \int_{\Re ( Z)>0} \operatorname{etr}( X Z) | Z|^{- b} &{}_{p}F_{q} \left(\begin{matrix}    a_{1}, \dots, a_{p}\\ b_{1}, \dots, b_{q}    	\end{matrix};  Z^{-1}\right) (d Z)\\		= {\varGamma_{n}(b)^{-1}2^{-n} }{(\pi i)^{\frac{n(n+1)}{2}}} | X|^{b - \frac{n+1}{2}} &{}_{p}F_{q+1} \left(\begin{matrix}    a_{1}, \dots, a_{p}\,\,\,\, \\ b_{1}, \dots, b_{q}, b     	\end{matrix};  X\right) 
        \end{aligned}
        \label{eq: gamma integral inv}
	\end{equation}
	%and two limit equalities in \cite{herz1955},
    %\begin{eqnarray}       \lim\limits_{\gamma\to \infty}{}_{p+1}F_{q} \left(\begin{matrix}a_{1}, \dots, a_{p}, \gamma \\ b_{1}, \dots, b_{q} \quad     	\end{matrix}; \gamma^{-1} T\right) = {}_{p}F_{q} \left(\begin{matrix}    	    a_{1}, \dots, a_{p}\\ b_{1}, \dots, b_{q}    	\end{matrix};  T\right), \\        \lim\limits_{\gamma\to \infty}{}_{p}F_{q+1}\left(\begin{matrix}  a_{1}, \dots, a_{p} \quad \\ b_{1}, \dots, b_{q}, \gamma     	\end{matrix}; \gamma  T\right)  = {}_{p}F_{q} \left(\begin{matrix}    	    a_{1}, \dots, a_{p}\\ b_{1}, \dots, b_{q}    	\end{matrix};  T\right),    \end{eqnarray}
    where both $\Re (a)$ and $\Re (b) > \frac{1}{2}(n-1)$, and $ Z$ is an $n \times n$ complex symmetric matrix such that the real part of the matrix $ Z$ is positive definite, denoted as $\Re( Z) >  0$. %In addition, the initial condition is    \begin{equation}    {}_{0}F_{0}( T) = \operatorname{etr}( T).    \end{equation}

The hypergeometric function of two matrix arguments of the same size is defined  as 
\begin{equation}
    \int_{O(n)} {}_{p}F_{q}\left(\begin{matrix}
    	    a_{1}, \dots, a_{p}\\ b_{1}, \dots, b_{q}
    	\end{matrix}; X H Y H^{\prime}\right) [dH] = {}_{p}F_{q}\left(\begin{matrix}
    	    a_{1}, \dots, a_{p}\\ b_{1}, \dots, b_{q}
    	\end{matrix}; X, Y\right),
\end{equation}
where $ X$ and $ Y$ are both $n \times n$ symmetric matrices, and the integral is taken over the space of $n\times n$ orthogonal matrices $O(n)$ with respect to the normalized Haar measure $[dH]$. If two matrices are of unequal size, for example, if $n \geq p$, and $A$ and $B$ $n\times n$ and $p\times p$ real symmetric matrices respectively,
 \[\begin{aligned}
        \int_{O(n)}{}_{p}F_{q}\left(\begin{matrix}
    	    a_{1}, \dots, a_{p}\\ b_{1}, \dots, b_{q}
    	\end{matrix}; A H_1 B H_1^{\prime}\right) [dH]& = {}_{p}F_{q}\left(\begin{matrix}
    	    a_{1}, \dots, a_{p}\\ b_{1}, \dots, b_{q}
    	\end{matrix}; A, B\right),\\
        H = [ H_{1},  H_{2}],&  \, H_1 \text{ is } n \times p,
    \end{aligned}\]
where the integral runs over the orthogonal group $O(n)$.

    \begin{lemma}[\cite{constantine1963some,constantine1966hotelling}]\label{lem: constantine} The incomplete Gamma and Beta integrals are 
    \begin{equation*}
        \begin{aligned}
                    \int_{0}^{ R} \operatorname{etr}(- A S)| S|^{a-\frac{p+1}{2}} d  S= \frac{| R|^a }{{B}_{p}\left(a,\frac{p+1}{2}\right)} {}_1F_{1} \left(a;a + \frac{p+1}{2}; -  A  R\right)\\
    \int_{0}^{ R} | S|^{a-\frac{p+1}{2}} | I -  S|^{b - \frac{p+1}{2}} d  S= \frac{| R|^a }{{B}_{p}\left(a,\frac{p+1}{2}\right)}{}_2F_{1} \left(a, - b + \frac{p+1}{2}; a + \frac{p+1}{2};  R\right)
        \end{aligned}
    \end{equation*}\end{lemma}
    
    \begin{lemma}[\cite{Davis1979InvariantPW,Davis1981OnTC}]\label{lem: davis}
    The incomplete Gamma and Beta integrals involving hypergeometric functions are 
    \begin{equation}
    \begin{aligned}
    \int_{0}^{ R} \operatorname{etr}(- A S)| S|^{c-\frac{p+1}{2}} {}_{p}F_{q}\left(\begin{matrix}
                a_{1},  \dots,a_{p}\\
                b_{1},  \dots,b_{q}
            \end{matrix}; B  S\right) d  S \\
                    = \frac{| R|^a }{{B}_{p}(c,\frac{p+1}{2})}
                    {}_{p+1}F_{q+1} \left(\begin{matrix}
                c, a_{1},  \dots,a_{p}\\
                c + \frac{p+1}{2},b_{1},  \dots,b_{q}
            \end{matrix}; -  A  R, B  R\right)
    \end{aligned}
    \end{equation}
    \begin{equation}
            \begin{aligned}
            \int_{0}^{ R} | S|^{c-\frac{p+1}{2}} | I -  S|^{d - \frac{p+1}{2}} {}_{p}F_{q}\left(\begin{matrix}
                a_{1},  \dots,a_{p}\\
                b_{1},  \dots,b_{q}
            \end{matrix}; A  S\right) d  S \\ = \frac{| R|^a }{{B}_{p}\left(c,\frac{p+1}{2}\right)} {}_{p+2}F_{q+1} \left(
            \begin{matrix}
                c, - d + \frac{p+1}{2}, a_{1},  \dots,a_{p}\\
                c + \frac{p+1}{2}, b_{1},  \dots,b_{q}
            \end{matrix};  A,  A  R\right)
        \end{aligned}
    \end{equation}
    \end{lemma}

\section{Matrix Normal Populations $T_{1}$, $T_{1\frac{1}{2}}$, $T_{2}$, and $T_{3}$}

Before going on, we will introduce four matrix normal distributions $X$, viewed as an $n \times p$ matrix with normal variables if $\operatorname{vec}(X')$, the usual vectorisation of $X'$, is the multivariate normal distribution $N_{np}(0,\varSigma)$. We classify the matrix normal populations according to four tensor decompositions of their precision matrix $\varTheta = \varSigma^{-1}$. To make this clear, block an $np\times np$ real symmetric positive definite matrix $\varTheta$ as
	\begin{equation}
		\varTheta = \begin{bmatrix}
			\varTheta_{11} &\varTheta_{12} & \dots & \varTheta_{1n}\\
			\varTheta_{21} & \varTheta_{22} & \dots & \varTheta_{2n}\\
			\vdots & \vdots & \ddots & \vdots \\
			\varTheta_{n1} & \varTheta_{n2} & \dots & \varTheta_{nn}
		\end{bmatrix}
		\label{eq: block matrix}
	\end{equation}
	where each $\varTheta_{ii}$ is positive definite of order $p$. Assume $\varTheta$ is one of the following four types.
	\begin{itemize}
		\item There exist orthonormal column vectors $\{b_{j}\}$ of length $p$ such that $\varTheta$ can be developed into the sum
		\begin{equation*}
		    \varTheta = \sum_{j=1}^{n}\sum_{j^{\prime}=1}^{n} A_{jj^{\prime}} \otimes B_{jj^{\prime}}, \eqno{(T_{1})}
		\end{equation*}
		where $B_{jj^{\prime}} = b_{i} \overline{b}_{j^{\prime}}^{\prime}$ and $A_{jj^{\prime}}(k,l) = \overline{b}_{j}^{\prime} \varTheta_{kl} b_{j^{\prime}}$. In this case, $A_{jj^{\prime}}(k,k)  = 0 $ if $j \neq j^{\prime}$, or in a word, $A_{jj^{\prime}}$ vanishing on diagonals if $j\neq j'$.
		\item It is $T_{1}$ and additionally, 
        \begin{equation*}
            A_{jj^{\prime}} = A_{jj}\delta_{jj^{\prime}}, \eqno{(T_{1\frac{1}{2}})}
        \end{equation*}
        that is, $A_{jj^{\prime}}$ vanishing if $j\neq j^{\prime}$.
		\item It is $T_{1}$ and there exist further orthonormal column vectors $\{a_{i}\}$ of length $n$ such that $\varTheta$ can be developed into the sum
		\begin{equation*}
		    \varTheta = \sum_{i=1}^{n}\sum_{j=1}^{p} \gamma_{ij} A_{i} \otimes B_{j}, \eqno{(T_{2})}
		\end{equation*}
		where $A_{i} = a_{i}\overline{a}_{i}^{\prime}$, $B_{j} = b_{j}\overline{b}_{j}^{\prime}$, and $\otimes$ is the Kronecker product operation.
		\item It is $T_{2}$ and additionally, there exist $\alpha_{i}, \beta_{j}$ such that $\gamma_{ij} = \alpha_{i} \beta_{j}$ or equivalently
	\begin{equation*}
	    \varTheta =  \varPhi^{-1} \otimes \varPsi^{-1}, \eqno{(T_{3})}
	\end{equation*}	
        where $\varPhi^{-1}  = \sum_{i=1}^{n} \alpha_{i} a_{i} \overline{a}_{i}^{\prime}$ and $\varPsi^{-1}  = \sum_{j=1}^{p} \beta_{j} b_{j}\overline{b}_{j}^{\prime}$.
	\end{itemize}

\begin{proposition}\label{prop: pdf} The probability density function $p(X)$ of a matrix normal population $X$ with its precision matrix $ \varTheta_{i}$ given by $T_{i}$, $i=1,1\frac{1}{2},2,3$ are

\begin{equation*}
                p(X) = \frac{| \varTheta_{1}|^{\frac{1}{2}}}{(2\pi)^{\frac{np}{2}}} 
            \operatorname{etr} \left(-\frac {1}{2} \sum_{j=1}^{p} { A}_{jj}  X{ B}_{jj}{ X}^{\prime} - \sum_{j=1}^{p}\sum_{j^{\prime}=j+1}^{p}{ A}_{jj^{\prime}}  X{ B}_{jj^{\prime}}{ X}^{\prime}\right),\eqno(T_1)
			\label{eq: T1 d}\\
\end{equation*}
\begin{equation*}
                p(X) = \frac{| \varTheta_{1\frac{1}{2}}|^{\frac{1}{2}}}{(2\pi)^{\frac{np}{2}}}\operatorname{etr} \left(-\frac {1}{2} \sum_{j=1}^{p} { A}_{jj} X{ B}_{jj}{ X}^{\prime}\right), \eqno(T_{1\frac{1}{2}})		\label{eq: T1.5 d}\\
\end{equation*}
\begin{equation*}
            p( X) =\frac{| \varTheta_{2}|^{\frac{1}{2}}}{(2\pi)^{\frac{np}{2}}}
             \operatorname{etr} \left(-\frac {1}{2} \sum_{i=1}^{n}\sum_{j=1}^{p} \gamma_{ij} { A}_{i} X{ B}_{j}{ X}^{\prime}\right),\eqno(T_2)
			\label{eq: T2 d}\\
\end{equation*}
\begin{equation*}
                p( X) = \frac{1}{(2\pi)^{\frac{np}{2}}| \varPhi|^{\frac{p}{2}}| \varPsi|^{\frac{n}{2}}}
               \operatorname{etr} \left(-\frac {1}{2}  \varPhi^{-1}  X  \varPsi^{-1} {{ X}}^{\prime}\right). \eqno(T_3)
				\label{eq: T3 d}
\end{equation*}
\end{proposition}

\begin{proposition}\label{prop: sd} Suppose $Z$ is a matrix population with precision matrix $\varTheta$.
        \begin{itemize}
	\item $Z$ is partially diagonalisable, say there exist orthonormal column vectors $\{ b_{j}\}$ of length $p$ such that $Z$ can be developed into the sum
	\[ Z = \sum_{j=1}^{p}  Z_{j}  \overline{b}_{j}^{\prime},\eqno(T_{1})\]
	where $ Z_{j} =  Zb_{j}$ and ${E}  Z_{j}(i) { \overline Z_{j^{\prime}}^{\prime}(i)} = c_{j}(i)\delta_{jj^{\prime}}$, or in a word, the random coefficients $Z_{j}$ are pairwise uncorrelated {at each row index} $i$ if and only if $\varTheta$ is $T_{1}$.
	\item $Z$ is {partial orthogonal diagonalisable}, say it is $T_{1}$ and the coefficients $ Z_{j}$ are {totally uncorrelated with each other}
    $${ E}  Z_{j}(i) \overline Z_{j^{\prime}}^{\prime}(i^{\prime}) = c_{j}(i,i^{\prime})\delta_{jj^{\prime}} \eqno(T_{1\frac{1}{2}})$$
    if and only if $\varTheta$ is $T_{1\frac{1}{2}}$.
	\item $Z$ is { totally orthogonal diagonalisable}, say it is $T_{1}$ and there exist further row vectors $\{ a_{i}\}$ of length $n$ such that $ Z$ can be developed into the sum
	\[ Z = \sum_{i=1}^n \sum_{j=1}^p \gamma_{ij}  a_{i} \overline {b}_{i}^{\prime},\eqno(T_{2})\]
	where $\gamma_{ij} = \overline{ a}_{i}^{\prime}  Z b_{j}$ and the random variables $\gamma_{ij}$ are {pairwise uncorrelated with each other} ${ E}\gamma_{ij} \overline{\gamma}_{i^{\prime}j^{\prime}} = c_{ij}\delta_{ii^{\prime}}\delta_{jj^{\prime}}$ if and only if $\varTheta$ is $T_{2}$.
	\item $Z$ is {totally diagonalisable}, say it is $T_{2}$ and there exist $\{\tau^2_{i}\}$ and $\{\sigma^2_{j}\}$ such that 
    $${ E} \gamma_{ij} \overline{\gamma}_{jj^{\prime}} =\sigma_{i}^{2}\tau_{j}^{2}\delta_{ii}\delta_{j^{\prime}j^{\prime}}\eqno(T_{3})$$
    if and only if $\varTheta$ is $T_{3}$.
    \item $Z$ is degenerate to rank one if and only if $\varTheta$ is $T_{3}$ and there exist random variables $\alpha_{i}, \beta_{j}$ such that $\gamma_{ij} = \alpha_{i} \beta_{j}$ and
	\[ Z = \left(\sum_{i=1}^{n} \alpha_{i}  a_{i}\right) \left(\sum_{j=1}^{p}  \beta_{j}  \overline{b}_{j}^{\prime}\right),\eqno(**)\]
	where ${E} \alpha_{i}\overline {\alpha}_{i^{\prime}} = \sigma_{i}^{2}\delta_{ii^{\prime}}$, ${E} \beta_{j}\overline {\beta}_{j^{\prime}} = \tau_{j}^{2}\delta_{jj^{\prime}}$.
        \end{itemize}
\end{proposition}

\begin{proof}[Proof of Propositions \ref{prop: pdf} and \ref{prop: sd}]
    To make this clear, let us vectorise an $n\times p$ matrix population $ X = [\textrm{x}_1; \textrm{x}_2; \dots; \textrm{x}_p]$ by row vectors as ${\operatorname{vec}} ( X') = (\textrm{x}_1, \textrm{x}_2, \dots, \textrm{x}_p)'$ and assume its covariance matrix $ \varSigma = \operatorname{Cov}({\operatorname{vec}} ( X'))$ exists and is positive definite. Then block the $np\times np$ precision matrix $\varTheta= \varSigma^{-1}$ according to \eqref{eq: block matrix}, where the diagonal blocks $ \varTheta_{ii}$($i=1,2,\dots,n$) are positive definite matrices of size $p$. Otherwise, $\operatorname{vec}(X')$ may fail to be a multivariate normal distribution. From the classical results of the spectral decomposition for positive definite matrices, we can read from Table \ref{tab: equivalent conditions} the nested diagonalisation properties. 

    \begin{table}[!h]
	\caption{Comparison of $T_{1},T_{1\frac{1}{2}},T_{2}$, $T_{3}$.}
	\centering
        \begin{tabular}{ccl}
	\label{tab: equivalent conditions}\\
        \hline
		$T_{1}$ & $\Leftrightarrow$  & $ \varTheta = \sum_{j=1}^{p}\sum_{j^{\prime}=1}^{p}  A_{jj^{\prime}} \otimes  B_{jj^{\prime}}$ where $ B_{jj^{\prime}}={ b_{j}} \overline b_{j^{\prime}}^{\prime}$.\\
		&  $\Leftrightarrow$   & $ B_{jj^{\prime}} = { b_{j}} \overline b_{j^{\prime}}^{\prime}$ consists of common eigenvectors ${ b_{j}}$ of $ \varTheta_{kk}$, \\ 
		& & independent of $k$, corresponding to eigenvalues $ A_{jj^{\prime}}(k,k)$.\\	
		\hline
		$T_{1\frac{1}{2}}$ & $\Leftrightarrow$ & $ \varTheta = \sum_{j=1}^{p}  A_{jj} \otimes  B_{jj}$ where $ B_{jj} = { b_{j}} \overline b_{j}^{\prime}$.\\
		&  $\Leftrightarrow$   & $ B_{jj} = { b_{j}} \overline b_{j}^{\prime}$ consists of common eigenvectors ${ b_{j}}$ of $ \varTheta_{kl}$, \\
		& & independent of $k,l$, corresponding to eigenvalues $ A_{jj}(k,l)$.\\	
		\hline 
		$T_{2}$ & $\Leftrightarrow$ & $ \varTheta = \sum_{i=1}^{n}\sum_{j=1}^{p} \gamma_{ij}  A_{i} \otimes  B_{j}$ where $ A_{i} =  a_{i}{ \overline a_{i}}^{\prime}$ and $ B_{j} = { b_{j}} \overline b_{j}^{\prime}$\\
		& $\Leftrightarrow$ & $ A_{i} =  a_{i}{\overline  a_{i}}^{\prime}$, $ B_{j} = { b_{j}} \overline b_{j}^{\prime}$ consists of eigenvectors $ a_{i}\otimes { \overline b_{j}}$ of $ \varTheta$,\\
		& & corresponding to eigenvalues $\gamma_{ij}$.\\
		\hline
		$T_{3}$ & $\Leftrightarrow$ & $ \varTheta=  \varPhi^{-1} \otimes  \varPsi^{-1}$ and $ \varPhi^{-1} = \sum_{i=1}^{n} \alpha_{i}  a_{i}{ \overline a_{i}}^{\prime}$, $ \varPsi^{-1} = \sum_{j=1}^{p} \beta_{j}{ b_{j}}  \overline b_{j}^{\prime}$.\\ 
		& $\Leftrightarrow$ & $ a_{i}\otimes{ \overline b_{j}}$ are eigenvectors of $ \varTheta$, corresponding to eigenvalues $\alpha_{i}\beta_{j}$.\\
		\hline		       
        (**) & $\Leftrightarrow$ & there exists column random vectors $X$ and $Y$ such that \\
        & & $Z = XY$ and $\varTheta= A \otimes B$ where $A = E X\overline{X}^{\prime}$, $B = E Y\overline{Y}^{\prime}$.\\ 
		\hline		
		\end{tabular}
    \end{table}

    Thus, Proposition \ref{prop: sd} is proved. Proposition \ref{prop: pdf} is a direct consequence from the fact that $\operatorname{tr} ( A X B{ X}') = {\operatorname{vec}( X')}^{\prime}( A \otimes  B) \operatorname{vec}( X' )$ and the assumption that $\operatorname{vec}(X')$ is multivariate normal.
\end{proof}

\section{The Non-central Distribution of Quadratic Forms}
In this and the following two sections, we will study the distribution of the largest and the smallest root of the sample covariance from a matrix normal population $T_1$, as a generalisation of incomplete Gamma integrals to the non-central settings.

\begin{theorem} \label{thm: main theorem}  Suppose $ X$ is an $n\times p$  real matrix according to the matrix normal distribution $T_1$ %with its precision matrix $ \varSigma_{1}$ partially diagonalisable in terms of an orthogonal matrix 
and $ B = ( b_{1}, b_{2},\dots, b_p)$,
%\[ \varSigma_{1} = \sum_{i=1}^{p} \sum_{j=1}^{p} A_{ij} \otimes  B_{ij},\quad   A_{ij} = ( I_{n}\otimes  b_{i})\Sigma_{1}( I_{n}\otimes  b_{i}^{\prime}),\quad  B_{ij} = \overline{ b_{i}}^{\prime} b_{j}.\]
Let $ M$ be a fixed real $n\times p$ matrix and $ S =  (X +  M)'(X +  M)$. Assume that the diagonals in the $p\times p$ real matrix $ U = (u_{ij})$ with $u_{ij} = \operatorname{tr}( A_{ij})$ are all positive, that is, $u_{ii}>0 \, (i=1,2,\dots,p)$. Then the probability density distribution of $ S$ depends only on $ T %=  B^{\prime} S B 
= (t_{ij})$ with 
    $t_{ij} = \operatorname{tr}( B_{ij}  S)$ when $n > p - 1$, and is 
        \begin{eqnarray}
         \frac{| \varTheta_{1}|^{\frac{1}{2}}}{2^{\frac{np}{2}}\varGamma_{p}(\frac{n}{2})} \operatorname{etr} \left(-\frac{1}{2}  U T \right) | T|^{\frac{n-p-1}{2}}          \text{ when } {M}=0; \\
                \times \operatorname{etr} \left(-\frac{1}{2}  \varOmega\right){}_{0}F_{1}\left(\frac{n}{2};\frac{1}{4} \varDelta  T\right) \text{ when } {M}\neq0;
        \end{eqnarray}
        where $\varOmega = \sum_{i,j=1}^{p} B_{ij}  M^{\prime}   A_{ij}  M$ and $ \varDelta = \sum_{i,j,k,l=1}^{p} B'_{ij} M'A'_{ij}A_{kl} M B_{kl}$.
\end{theorem}

\begin{proof}[Proof of Theorem \ref{thm: main theorem}]
        Suppose the central part holds for $n > p-1$. Decomposing $ X =  H  Z$ where $ Z$ is upper-triangular and $ H' H =  I_p$. Extend $H$ to an $n\times n$ orthogonal matrix $K = [H,H_\perp]$ so that
    %From the known results such as Example 1.6 in \cite{1990Generalized} of Theorem 2.1.14 of \cite{muirhead1982aspects}, we have
	   \begin{equation*}
	       \begin{aligned}
		(dX) = 2^{-p} |Z'Z|^{\frac{n-p-1}{2}} \cdot (dZ'Z) \cdot (dK).%, \\(dK) = {(H'dH)} \cdot {(H_\perp'dH)} 
		\end{aligned}
	   \end{equation*}
    Thereby, we could rewrite \eqref{eq: T1 d} according to \cite{james1961zonal} as 
    \[\begin{aligned}
        \operatorname{etr} \left(-\frac{1}{2}  \varOmega\right)\int_{O(n)}\operatorname{etr}\left(-\frac{1}{2} \sum_{i,j=1}^{k}( A_{ij} M B_{ij})Z'H' \right) (d K) \\
        = \operatorname{etr} \left(-\frac{1}{2}  \varOmega\right){}_{0}F_{1}\left(\frac{n}{2};\frac{1}{4}\sum_{i,j,k,l=1}^{k}( B'_{ij} M'A'_{ij}A_{kl} M B_{kl}) Z'Z\right).    \end{aligned}\]
    %where we use the fact that these $ B_{ij}$ are orthogonal to each other,\begin{equation}   \begin{aligned}B_{ij}  B_{kl} =   B_{il} \delta_{jk}, \quad   B_{ij}^{\prime} =   B_{ji}, \quad   A_{ij}^{\prime} =   A_{ji}.    \end{aligned}    \label{eq: commutativity}    \end{equation}
    Thus, if the central part is true, Theorem \ref{thm: main theorem} is consequently proved by integrating over the orthogonal group $O(n)$. 
    
    From the result of \cite{khatri1966} on the quadratic forms in normal vectors, we can reduce the central part of Theorem \ref{thm: main theorem} to the $\frac{1}{2}p(p+1)$ independent quadratic forms by introducing $ Y =  X B = ( y_{1}, y_{2}\dots,  y_{p})$
    \begin{equation*}
    \begin{aligned}
        {y}_1'{A}_{11}{y}_1, {y}_1'{A}_{12}{y}_2, \dots, {y}_1'{A}_{1p}{y}_p,\\
        {y}_2'{A}_{22}{y}_2, \dots, {y}_2'{A}_{2p}{y}_p, \\
        \vdots\qquad \\
        {y}_p'{A}_{pp}{y}_p,
    \end{aligned}
\end{equation*}
In fact, we have for each ${y}_i'{A}_{ij}{y}_j$ the contribution to the probability density function
\[\begin{aligned}
    {_{0}}F_{0}\left( I - \frac{q_{ij}}{2} A_{ij}, q_{ij}^{-1} y_{i} y_{j}^{\prime}\right) = {_{0}}F_{0}\left( I - \frac{q_{ij}}{2} A_{ij}, q_{ij}^{-1} y_{j}^{\prime} y_{i}\right).\end{aligned}\]
However, we may find $ y_{j}^{\prime} y_{i} = t_{ij} = \operatorname{tr} ( B_{ij} S)$ so that from these properties of hypergeometric functions
\[\begin{aligned}
    {}_{0}F_{0} ( X,c Y) = {}_{0}F_{0} (c X, Y)\\
    {}_{0}F_{0} ( X,  I) = {}_{0}F_{0} ( X) = \operatorname{etr}( X), 
\end{aligned}\]
the terms concerning $q_{ij}$ cancel out. This yields the desired form.
\end{proof}

\section{The Latent Roots with Unequal Covariances}

The latent roots with unequal precision matrices $\varTheta_1 \neq \varTheta_2$ are fundamental to the determination of the null distribution in the Behrens-Fisher problem. %This theorem is motivated by the work of \cite{constantine1966hotelling}.

\begin{theorem} Let $X_{1}, X_{2}$ be two matrix normal population $T_{1}$ with the common normalising matrix $B$ and $\varTheta_{1}, \varTheta_{2}, U_{1},U_{2}$ defined similarly as Theorem \ref{thm: main theorem}. Let $S_{1} = X_1^{\prime}X_1, S_{2} = X_2^{\prime}X_2$ and $S = S_1S_2^{-1}$.  Then the probability density function for $S$ is
    \begin{equation}
    \begin{aligned}
    p( S) = \frac{| \varTheta_{1}|^{\frac{1}{2}}| \varTheta_{2}|^{\frac{1}{2}}\varGamma_p(\frac{n_1+n_2+p+1}{2})}{2^{(n_1 + n_2)p}{B}_{p}(\frac{n_1}{2},\frac{n_2}{2})\varGamma_p(\frac{p+1}{2})|U_2|^{\frac{n_1+n_2}{2}}}|S|^{\frac{n_1-p-1}{2}}\\
    \times
    {_{1}}F_{1}\left(\frac{n_1 + n_2}{2}; \frac{p+1}{2};-\frac{1}{2}  (U_1U_2^{-1} + S)\right),
\end{aligned}
\label{eq: S1S0-1 unequal}
\end{equation}
\end{theorem}

\begin{proof}[Proof of Theorem \ref{eq: S1S0-1 unequal}]
    In general, in order to determine the central distribution of $ S =  S_{1} S_{2}^{-1}$, one should observe a fact that in our case the mean is zero. A direct computation yields the joint distribution of $ S_{1} =  B' T_1  B$ and $ S_{2} =  B' T_2  B$ is
\begin{equation}
    \begin{aligned}
    p( S_1, S_2) = \frac{| \varTheta_{1}|^{\frac{1}{2}}| \varTheta_{2}|^{\frac{1}{2}}}{2^{(n_1+n_2)p}\varGamma_{p}(\frac{n_1}{2})\varGamma_{p}(\frac{n_2}{2})}\operatorname{etr} \left(-\frac{1}{2}   U_{2}  T_{2}\right)| T_{2}|^{\frac{n_2-p-1}{2}}\\
    \times \operatorname{etr} \left(-\frac{1}{2} U_1  T_1 \right)| T_1|^{\frac{n_1-p-1}{2}}.
\end{aligned}
\label{eq: S1 and S0 unequal a}
\end{equation}
The transformation $( S_{1}, S_{2}) \mapsto ( S_{1} S_{2}^{-1}, S_{2})$ has the Jacobian $| S_{2}|^{-{(p+1)}/{2}}$. In this situation, the joint distribution of $ S$ and $ S_{2}$ derived from \eqref{eq: S1 and S0 unequal a} is
\begin{equation}
    \begin{aligned}
     p( S , S_2) = \frac{| \varTheta_{1}|^{\frac{1}{2}}| \varTheta_{2}|^{\frac{1}{2}}}{2^{(n_1+n_2)p}\varGamma_{p}(\frac{n_1}{2})\varGamma_{p}(\frac{n_2}{2})} \operatorname{etr} \left(-\frac{1}{2}  (U_2 + U_1S)T_{2}\right)\\
     \times | S|^{\frac{n_1-p-1}{2}}| T_{2}|^{\frac{n_1+n_2 - p- 1}{2}}.
\end{aligned}
\label{eq: S and S0 a}
\end{equation}
By the Beta integral and these properties of hypergeometric functions
\begin{equation*}
    \begin{aligned}
        {_{1}}F_{1}(a;a; X) = {_{0}}F_{0}( X) = \operatorname{etr} ( X), \\
        \int_{ S> 0}| S|^{a-\frac{p+1}{2}}|I- S|^{b-\frac{p+1}{2}}{_{0}}F_{0}(R S)dS = {\rm B}(a,b)^{-1}{_{1}}F_{1}(a;b; R),
    \end{aligned}
\end{equation*}
integrating with respect to $ S_{2}$ in \eqref{eq: S and S0 a}, we have the distribution of $ S$,
\begin{equation}
    \begin{aligned}
    p( S) = \frac{| \varTheta_{1}|^{\frac{1}{2}}| \varTheta_{2}|^{\frac{1}{2}}\varGamma_p(\frac{n_1+n_2+p+1}{2})}{2^{(n_1 + n_2)p}{B}_{p}(\frac{n_1}{2},\frac{n_2}{2})\varGamma_p(\frac{p+1}{2})|U_2|^{\frac{n_1+n_2}{2}}}|S|^{\frac{n_1-p-1}{2}}\\
    \times
    {_{1}}F_{1}\left(\frac{n_1 + n_2}{2}; \frac{p+1}{2};-\frac{1}{2}  (U_1^{-1}U_2 + S)\right).
\end{aligned}
\label{eq: S1S0-1 unequal a}
\end{equation}
%If $ S =  S_1 ( S_1+ S_2)^{-1}$, then simple calculus yields the Jacobian of $S \mapsto S(I+S)^{-1}$ is $|I + S|^{-(p+1)}$ and 
%\begin{equation}   \begin{aligned}    p( S) = \frac{| \Sigma_1|}{2^{(n_1 + n_2)p}B_{p}(\frac{n_1}{2},\frac{n_2}{2})| U|^{\frac{n_1+n_2}{2}}}| S|^{\frac{n_1-p-1}{2}} | I -  S|^{\frac{-n_1}{2}}\\    \times     {_{1}}F_{1}\left(\frac{n_1 + n_2}{2}; \frac{p+1}{2}; \frac{1}{2} (U+S(I-S)^{-1})\right),\end{aligned}\label{eq: S1S-1}\end{equation}
\end{proof}

\section{The Non-Central Latent Roots with Unknown Covariance}
In contrast, the non-central latent roots with unknown (known $A_{ij}$ but unknown $B_{ij}$) precision matrix $\varTheta$ are used to derive the alternative distribution in multivariate analysis of variance. %This theorem can be compared to the result of \cite{james1964distribution}.
\begin{theorem} Let  $ X_{1}$ and $ X_{2}$ be two independent matrix normal populations $T_{1}$ with common precision matrix $\varTheta$, $M$ an arbitrary fixed $n_1\times p$ real matrix, $S_{1} = ( X_1 +  M)'( X_1 + M)$, and $ S_{2} =  X_2' X_2$. 
%The non-central distribution of the ratio $ S =  S_1( S_1+ S_2)^{-1}$ is 
%    \[    \begin{aligned}   \frac{| \Sigma_1|}{2^{(n_1 + n_2)p}B_{p}(\frac{n_1}{2},\frac{n_2}{2})| U|^{\frac{n_1+n_2}{2}}}| S|^{\frac{n_1 - p -1}{2}} | I -  S|^{\frac{n_2-p-1}{2}}\\    \times     {_{1}}F_{1}\left(\frac{n_1 + n_2}{2}; \frac{n_1}{2}; \Delta  U^{-1} S\right),\end{aligned}\]where $ \Omega = \sum_{i,j=1}^{p} B_{ij}  M^{\prime}   A_{ij}  M$ and $ \Delta = \sum_{i,j,k=1}^{p} B_{jk} M' A_{ik} A_{ij}^{\prime} M$.
The joint distribution of latent roots $f_{1},f_{2},\dots, f_{p}$ of $ S =  S_1( S_1+ S_2)^{-1}$ is 
    \begin{equation}
            \begin{aligned}
    \frac{\pi^{p^{2}/2}| \varTheta|}{2^{(n_1 + n_2)p}B_{p}(\frac{n_1}{2},\frac{n_2}{2})\varGamma_{p}(\frac{p}{2})| U|^{\frac{n_1+n_2}{2}}}\prod_{i<j}^{p} (f_{i} - f_{j}) {| F|^{\frac{n_1}{2}}}{| I -   F|^{\frac{n_1 -p-1}{2}} }\\
    \times \operatorname{etr} \left(-\frac{1}{2}  \varOmega\right)
    {_{1}}F_{1}\left(\frac{n_1 + n_2}{2}; \frac{n_1}{2}; \Delta  U^{-1},  F\right),
\end{aligned}
    \end{equation}
where $\varOmega $ and $ \varDelta$ are defined in Theorem \ref{thm: main theorem}, $F = \operatorname{diag}(f_{1},f_{2},\dots,f_{p}), f_{1}>f_{2}>\dots > f_{p}$; elsewhere zero.

In particular, the probability distribution function for the largest root $f_1$ is 
             \begin{equation}
                 \begin{aligned}
                {P}(f_1 < x) = \frac{| \varTheta|x^{\frac{n_1p}{2}}}{2^{(n_1 + n_2)p}B_{p}(\frac{n_1}{2},\frac{n_2}{2})B_{p}(\frac{n_1}{2},\frac{p+1}{2})| U|^{\frac{n_1+n_2}{2}}} \\\times  \operatorname{etr} \left(-\frac{1}{2}  \varOmega\right){_{2}}F_{1}\left(a, b; c; \Delta  U^{-1},  \Delta  U^{-1}  R\right).\end{aligned}
             \end{equation}
                where $a = -\frac{1}{2}(n_2 - p -1), b = \frac{1}{2}(n_1 + n_2)$, $c =  \frac{1}{2}(n_2 + p +1)$.                Similarly, $1$ $-$ probability distribution function for the smallest root $f_p$ is
                \begin{equation}
                    \begin{aligned}
                 1 - {P}(f_p \leq y) = P(f_p > y)  = \frac{| \varTheta|(1-y)^{\frac{n_2p}{2}}}{2^{(n_1 + n_2)p}B_{p}(\frac{n_1}{2},\frac{n_2}{2})B_{p}(\frac{n_2}{2},\frac{p+1}{2})| U|^{\frac{n_1+n_2}{2}}} \\\times  \operatorname{etr} \left(-\frac{1}{2}  \varOmega\right){_{2}}F_{1}\left(a,b;c; \Delta  U^{-1},  \Delta  U^{-1}  R\right).\end{aligned}
                \end{equation}
                where $a = -\frac{1}{2}(n_1 - p -1), b = \frac{1}{2}(n_1 + n_2)$, $c =  \frac{1}{2}(n_1 + p +1)$.
\label{thm: non-central means}
\end{theorem}

\begin{proof}[Proof of Theorem \ref{thm: non-central means}]
    In general, in order to determine the non-central distribution of $ S =  S_{1} S_{2}^{-1}$, one should observe the fact that in our case $ S_{2}$ is assumed central. A direct computation yields the joint distribution of $ S_{1} =  B' T_1  B$ and $ S_{2} =  B' T_2  B$ (module $\operatorname{etr} \left(-\frac{1}{2}  \varOmega\right)$) is
\begin{equation}
    \begin{aligned}
    p( S_1, S_2) = \frac{| \varTheta |}{2^{(n_1+n_2)p}\varGamma_{p}(\frac{n_1}{2})\varGamma_{p}(\frac{n_2}{2})}\operatorname{etr} \left(-\frac{1}{2}   U  T_{2}\right)| T_{2}|^{\frac{n_2-p-1}{2}} \\
    \times \operatorname{etr} \left(-\frac{1}{2} U  T_1 \right)| T_1|^{\frac{n_1-p-1}{2}}{_{0}}F_{1}\left(\frac{n_1}{2}; \varDelta  T_1 \right).
\end{aligned}
\label{eq: S1 and S0 a2}
\end{equation}
The transformation $( S_{1}, S_{2}) \mapsto ( S_{1} S_{2}^{-1}, S_{2})$ has the Jacobian $| S_{2}|^{-{(p+1)}/{2}}$. In this situation, the joint distribution of $ S$ and $ S_{2}$ derived from \eqref{eq: S1 and S0 a2} is
\begin{equation}
    \begin{aligned}
     p( S , S_2) = \frac{| \varTheta |}{2^{(n_1+n_2)p}\varGamma_{p}(\frac{n_1}{2})\varGamma_{p}(\frac{n_2}{2})} \operatorname{etr} \left(-\frac{1}{2}  U( S + I)  T_{2}\right)\\ \times | S|^{\frac{n_1-p-1}{2}} | T_{2}|^{\frac{n_1+n_2 - p- 1}{2}}{_{0}}F_{1}\left(\frac{n_1}{2}; \varDelta  S  T_{2}\right).
\end{aligned}
\label{eq: S and S0 a2}
\end{equation}
By Lemma 3 in I and these properties of hypergeometric functions
\begin{equation*}
    \begin{aligned}
        {_{1}}F_{1}(a;a; X) = {_{0}}F_{0}( X) = \operatorname{etr} ( X), \\
        \int_{ S> 0} \operatorname{etr} (-  A S)| S|^{a-\frac{n+1}{2}}{_{0}}F_{1}(b; B S)dS = \varGamma_{n}(a)| A|^{-a}{_{1}}F_{1}(a;b; B A^{-1}),
    \end{aligned}
\end{equation*}
integrating with respect to $ S_{2}$ in \eqref{eq: S and S0 a2}, we have the distribution of $ S$,
\begin{equation}
    \begin{aligned}
    p( S) = \frac{| \varTheta|}{2^{(n_1 + n_2)p}B_{p}(\frac{n_1}{2},\frac{n_2}{2})| U|^{\frac{n_1+n_2}{2}}}| I +   S|^{-\frac{n_1 + n_2}{2}} | S|^{\frac{n_1-p-1}{2}}\\
    \times 
    {_{1}}F_{1}\left(\frac{n_1 + n_2}{2}; \frac{n_1}{2}; \varDelta  U^{-1}( I+  S^{-1} )^{-1}\right),
\end{aligned}
\label{eq: S1S0-1 a2}
\end{equation}
If $ S =  S_1 ( S_1+ S_2)^{-1}$, then simple calculus yields the Jacobian of $S \mapsto S(I+S)^{-1}$ is $|I + S|^{-(p+1)}$ and 
\begin{equation}
    \begin{aligned}
    p( S) = \frac{| \varTheta|}{2^{(n_1 + n_2)p}B_{p}(\frac{n_1}{2},\frac{n_2}{2})| U|^{\frac{n_1+n_2}{2}}}| S|^{\frac{n_1-p-1}{2}} | I -  S|^{\frac{n_2-p-1}{2}}\\
    \times 
    {_{1}}F_{1}\left(\frac{n_1 + n_2}{2}; \frac{n_1}{2}; \varDelta  U^{-1} S\right),
\end{aligned}
\label{eq: S1S-1 a2}
\end{equation}

(1) Non-central latent roots. By integrating over the orthogonal group, we obtained the joint distribution of latent roots of the ratio $ S =  S_1 ( S_1+ S_2)^{-1}$

\begin{equation}
    \begin{aligned}
    \frac{\pi^{p^{2}/2}| \varTheta|}{2^{(n_1 + n_2)p}B_{p}(\frac{n_1}{2},\frac{n_2}{2})\Gamma_{p}(\frac{p}{2})| U|^{\frac{n_1+n_2}{2}}}\prod_{i<j}^{p} (f_{i} - f_{j}) {| F|^{\frac{n_1-p-1}{2}}}{| I -   F|^{\frac{n_1 -p-1}{2}} } \\
    \times {_{1}}F_{1}\left(\frac{n_1 + n_2}{2}; \frac{n_1}{2}; \varDelta  U^{-1},  F\right),
\end{aligned}
\label{eq: non-central means}
\end{equation}
where $ F = \operatorname{diag}(f_{1},f_{2},\dots,f_{p}), f_{1}>f_{2}>\dots > f_{p}$; elsewhere zero.

(2) The largest and the smallest latent root $f_1$ and $f_p$.
            Based on the non-central distribution, we are going to evaluate this integral 
            \[\begin{aligned}
                {P}( S <  R) = \frac{| \varTheta|}{2^{(n_1 + n_2)p}B_{p}(\frac{n_1}{2},\frac{n_2}{2})| U|^{\frac{n_1+n_2}{2}}}\int_{ 0}^{ R} | S|^{\frac{n_1-p-1}{2}} | I -  S|^{\frac{n_2-p-1}{2}}\\
    \times 
    {_{1}}F_{1}\left(\frac{n_1 + n_2}{2}; \frac{n_1}{2}; \Delta  U^{-1} S\right) d  S\\
     = \frac{| \varTheta|| R|^{\frac{n_1}{2}}}{2^{(n_1 + n_2)p}B_{p}(\frac{n_1}{2},\frac{n_2}{2})B_{p}(\frac{n_1}{2},\frac{p+1}{2})| U|^{\frac{n_1+n_2}{2}}} \\
     \times {_{2}}F_{1}\left(a, b; c; \Delta  U^{-1},  \Delta  U^{-1}  R\right).\end{aligned}\]
                where $a = -\frac{1}{2}(n_2 - p -1), b = \frac{1}{2}(n_1 + n_2)$, $c =  \frac{1}{2}(n_2 + p +1)$. Thereby, by setting $R = xI_{m}$ so that $f_{1} < x$ equivalent to $S < xI_{m}$, the probability distribution function for $f_1$ is 
             \[\begin{aligned}
                {P}(f_1 < x) = \frac{| \Sigma_1|x^{\frac{n_1p}{2}}}{2^{(n_1 + n_2)p}B_{p}(\frac{n_1}{2},\frac{n_2}{2})B_{p}(\frac{n_1}{2},\frac{p+1}{2})| U|^{\frac{n_1+n_2}{2}}} \\\times  {_{2}}F_{1}\left(a, b; c; \Delta  U^{-1},  \Delta  U^{-1}  R\right).\end{aligned}\]
                where $a = -\frac{1}{2}(n_2 - p -1), b = \frac{1}{2}(n_1 + n_2)$, $c =  \frac{1}{2}(n_2 + p +1)$.
Similarly, the smallest latent root $f_p$ is the largest latent root of $ I -  S$. By symmetry, we have the probability distribution function for $f_1$.
\end{proof}

\section{Conclusion}

In this work, we have derived the joint distribution of all $\frac{1}{2}p(p+1)$ quadratic forms $y_i' A_{ij} y_j$, extending classical results of non-central Wishart distribution and normal quadratic forms. Based on incomplete Gamma and Beta integrals involving hypergeometric functions of matrix arguments, we reconsider multivariate analysis of variance in three settings (A), (B), and (C). 

By expressing these forms through the orthogonal decomposition $Y = XB'$ and introducing the Kronecker‐sum matrix $\varTheta$, we showed that, under $n>p-1$, their joint density admits a compact form, similar to the Wishart density or the multivariate Gamma function
$$
\frac{|\varTheta|^{\frac{1}{2}}}{2^{\frac{np}{2}} \Gamma_p(\frac{n}{2})}\operatorname{etr}\left(-\frac{1}{2}UT\right)\,|T|^{\frac{n-p-1}{2}},
$$
where $U=(\mathrm{tr} A_{ij})$ and $T=(b_i'Sb_j)$. Also, the derivation of such distributions facilitates the study of eigenvalue distributions in non-central settings.  Our stochastic decomposition $T_1$ further clarifies when $Z$ can be partially diagonalised, enabling tractable analysis of latent roots. These formulas pave the way for asymptotic investigations, such as Tracy–Widom limits of the largest latent root in high‐dimensional regimes, and may inform future applications in multivariate inference and random matrix study.

%The author thanks Jia-Juan Liang, Kai-Tai Fang, and Dietrich von Rosen who carefully read this manuscript.

%The author thanks Kai-Tai Fang for his encouragement and discussion during the preparation of this paper. It is his classical book that leads me into the field of multivariate statistics. However, the classification of general normal populations is far from complete. Here, this article only tries to throw a pebble in this way. %The author should also thank Jia-Juan Liang and the editor Dietrich von Rosen, who read this manuscript very carefully and gave so many valuable suggestions.
%The author thanks Professor Kai-Tai Fang for discussion and help.

\bibliographystyle{abbrvnat}
\bibliography{bibtex}

%\begin{figure}[!h]
%    \centering
%    \begin{subfigure}[b]{0.7\textwidth}
 %       \includegraphics[width=\textwidth]{Rplot01.png}
  %      \caption{$n = 200, p = 4, \gamma = 1$.}
   %     \label{fig:sub1}
   % \end{subfigure}
   % \vspace{0.5cm} % 垂直间距
   % \begin{subfigure}[b]{0.7\textwidth}
   %     \includegraphics[width=\textwidth]{Rplot02.png}
   %     \caption{$n = 200, p = 4, \gamma = 2$.}
   %     \label{fig:sub2}
   % \end{subfigure} 
   % \vspace{0.5cm} % 垂直间距
   % \begin{subfigure}[b]{0.7\textwidth}
   %     \includegraphics[width=\textwidth]{Rplot04.png}
   %     \caption{$n = 200, p = 4, \gamma = 4$.}
   %     \label{fig:sub4}
   % \end{subfigure} 
   % \caption{Exact power analysis for $\mu = \mu_1 - \mu_2$ ($ \varPhi = \gamma  I$, $ \varPsi=  I$).}
%\end{figure}

\end{document}